\documentclass{article}%
\usepackage{amsmath}
\usepackage{amsfonts}
\usepackage{amssymb}
\usepackage{amsthm}
\usepackage{graphicx}%
\providecommand{\U}[1]{\protect\rule{.1in}{.1in}}
\DeclareMathOperator{\Alt}{Alt}
\DeclareMathOperator{\Sym}{Sym}
\DeclareMathOperator{\SL}{SL}
\DeclareMathOperator{\GL}{GL}

\DeclareMathOperator{\Id}{Id}
\DeclareMathOperator{\supp}{supp}
\newtheorem{theorem}{Theorem}

\newtheorem{lemma}[theorem]{Lemma}
\newtheorem{proposition}[theorem]{Proposition}
\newtheorem{remark}{Remark}
\newcommand{\N}{\mathbb{N}}
\newcommand{\F}{\mathbb{F}}

\begin{document}
\title{Words with few values in finite simple groups}
\author{M. Kassabov and N. Nikolov}
\maketitle

\begin{abstract} We construct words with small image in a given finite alternating or unimodular group. This shows that word width in these groups is unbounded in general.
\end{abstract}

Let $w$ be a group word, i.e., an element of the free group on $x_1, \ldots, x_d$. For a group $G$ we denote the set of values of $w$ by
$G_w:=\{ w(g_1, \ldots, g_d)^{\pm 1} \ | \ g_i \in G\}$
and the verbal subgroup $\langle G_w\rangle$ is $w(G)$.

The study of the images of the word maps dates back to the theory of varieties of groups (see~\cite{HN}) and the work of P.~Hall and his students, for a modern exposition also see~\cite{dan}.

Recently, there has been a lot of progress in understanding $G_w$ when $G$ is a finite group, in particularly a finite simple group. In~\cite{LOST} it was proved that when $G$ is a finite simple group then every element of $G$ is a commutator and in addition every element of $G$ is a product of two squares~\cite{squares}.
More generally~\cite{LS} shows that for any $w \not = 1$ we have $G=G_wG_w$ when $G$ is a sufficiently large finite simple group. In this note, we show that the requirement of the size of $G$
can not be removed, even if we only require that $G = (G_w)^k$ for a fixed $k$.

\begin{theorem}
\label{nobound}
For any $k$ there exist a word $w$ and a finite simple group $G$,
such that $w$ is not an identity in $G$, but $G \not = (G_w)^k$.
\end{theorem}

We obtain this as an immediate corollary of the following results about alternating groups and
special linear groups.


\begin{theorem}
\label{alt}
For every $n\geq 7, n \not =13$ there is a word $w(x_1,x_2) \in F_2$ such that
$\Alt(n)_w$ consists of the identity and all $3$-cycles. When $n=13$ there is a word
$w(x_1,x_2,x_3) \in F_3$ with the same property.
\end{theorem}

Of course the same result holds for $\Alt(5)$ by taking $w=x_1^{10}$.
However, note that $\Alt(6)$ is a genuine exception,
because it has an outer automorphism which send the $3$-cycles to the double $3$-cycles.

Similar result holds for $\Sym(n)$. In fact with the exception of $\Sym(7)$ the words constructed in Theorem~\ref{alt}
also satisfy $\Sym(n)_w = \Alt(n)_w$.

We obtain similar result for the groups $\SL_n(q)$.

\begin{theorem}
\label{SL}
For every $n,q \geq 2$ with the possible exception of $\SL_4(2)$ there is a word $w(x_1,x_2) \in F_2$ such that
$\SL_n(q)_w$ consists of the identity and the conjugacy class of all transvections.

For $\SL_4(2)$ the word $w=x_1^{2.3.5.7}$ takes values the identity, the transvections and the double transvections with Jordan normal form $J_2(1) \oplus J_2(1)$.
\end{theorem}

\medskip

\begin{proof}[Proof of Theorem~\ref{nobound}]
This is an immediate consequence of Theorem~\ref{alt}, because for $n > 2k +1$ some elements
in $\Alt(n)$ can not be written as product of less than $k+1$ $3$-cycles.
\end{proof}

\section{Proof of Theorem~\ref{alt}}

Everywhere in this section $n$ is an integer bigger than $6$. Assume first that $n \not =7,13$.

\begin{proposition}
\label{pirmes}
For each $n>7, n \not =13$ there exist some $k \geq 1$ and some prime numbers
$$
p_1>p_2> \cdots > p_k>3
$$
such that $n-\sum_i p_i \in \{3,4,5\}$ and $p_i> (n- \sum_{j=1}^{i-1}p_j)/2$
for each $i=1, \ldots, k$.
\end{proposition}
\begin{proof}
Indeed take a prime $p_1 \in (n/2, n-3]$ distinct from $n-6,n-7,n-13$ and continue by induction.
By inspection such a prime always exists for all $n\leq 50$ and for $n>50$ there is always a prime
in the interval $(n/2,n-14]$. Indeed by~\cite{Loo} there is always a prime in the interval $(3m,4m)$ for $m \in \N$, $m>1$. Now take $m$ so that $m-1 < n/6 \leq m$ and than any prime in $(3m,4m)$ will do since $4m < 2n/3 + 4 \leq n-13$ as $n \geq 51$.
\end{proof}

\medskip

Let $M$ be the exponent of $\Alt(n)$ and write $M$ as a product of prime powers $M= \prod_{p \leq n} p^{l_p}$. Let $k$ and $p_1, \ldots, p_k$ be as given in Proposition \ref{pirmes}. Define $m_i=M/p_i^{l_{p_i}}$ for $i=1,2, \dots,  k$.
We set $p_0=3$ and $m_0:=M/3^{l_3}$.

Consider now the word $w_1$ in $x_1,x_2$ defined by the left normed commutator
\[
w_1(x_1,x_2)=[(x_1^{m_0})^{[\cdots [[x_2,x_1^{m_1}],x_1^{m_2}], \cdots,x_1^{m_k}]},x_1^{m_0}].
\]
The first result we need is

\begin{lemma}
\label{values}
The word $w_1(x_1,x_2)$ takes some $3$-cycle as a value in $\Alt(n)$.
\end{lemma}

\begin{proof}
When $x_1$ is a product of disjoint  $p_i$-cycles, one for each $i$ then each $x_i^{m_i}$ is a $p_i$-cycle.
Hence it is enough to prove that there is some collection of disjoint $p_i$-cycles $a_i$ for $i=0, 1, \ldots, k$ and an element $y$ such that $[a_0^{[y,a_1, \ldots, a_k]},a_0]$ is a 3-cycle.

First observe that from the definition of $p_i$ we have that each
$p_i> (n-\sum_{j=1}^{i-1}p_j)/2$.

We define $y'$ to be any $n$-cycle on $\{1, \ldots n\}$. Note that $y'$ may not be even but if needed we will modify it with a transposition at the end.

Put $\Omega_0=\{1, \ldots, n\}$ and $v_0=y'$.
Assume by induction that for some $1 \leq i<k$ the set $\Omega_{i-1}$ is a segment of a cycle of $v_{i-1}$, i.e., $v_{i-1}$ acts on the elements $b_1,b_2, \ldots,..b_t$ of $\Omega_{i-1}$ as $b_j^{v_{i-1}}=b_{j+1}$ for $j=1, \ldots, t-1$, where $t = n -\sum^{i-1}_{j=1}p_j$. Note that we don't specify the image of $b_t$ and the preimage of $b_1$ under $v_{i-1}$,

This is clearly true for $i=1$.

Recall that the support $\supp (\pi)$ of a permutation $\pi$ is the set of points moved by $\pi$. For $i=1, \ldots, k-1$ define inductively the $p_i$-cycle $a_i$ such that
\begin{itemize}
\item
$\supp(a_i) \subset \Omega_{i-1}$ and
\item
The set $\Omega_{i-1} \backslash \supp(a_i)$ is a segment of a cycle of $[v_{i-1},a_i]$.
\end{itemize}

Indeed suppose that the elements of $\Omega_{i-1}$ are $\{b_1, \ldots b_t\}$ where $t=|\Omega_{i-1}|$ where $v_{i-1}$ sends $b_j$ to $b_{j+1}$ for $j=1,2, \ldots, t-1$. Set $t'=t-p_i$. Define the $p_i$-cycle $a_i$ to be any $p_i$-cycle on $\Omega_{i-1} \backslash \{b_2,b_4, \ldots, b_{2t'}\}$ which begins with
$$
a_i: \quad b_1 \mapsto b_3 \mapsto b_5 \mapsto \ldots \mapsto b_{2t'+1}
$$
Since  $t'< t/2$ we can always find such a $p_i$-cycle. Put $\Omega_i= \{b_2,b_4, \ldots, b_{2t'}\}= \Omega_{i-1} \backslash \supp(a_i)$. Now $a_i^{v_{i-1}}$ is a $p_i$-cycle which acts on $\Omega_i$ as $b_2 \mapsto b_4 \mapsto \cdots \mapsto b_{2t'}$ and therefore
$[v_{i-1},a_i]=a_i^{-v_{i-1}}a_i$ acts on $\Omega_i$ as $b_{2t'} \mapsto b_{2t'-2} \mapsto \cdots \mapsto b_2$. This justifies that we can make the choice of $a_i$ as claimed.

Next set $v_i:=[v_{i-1},a_i], \Omega_i= \Omega_{i-1} \backslash \mathrm{supp}(a_i)$ and continue by induction.
We reach an element $v_{k-1}=[y,a_1, \ldots, a_{k-1}]$ acting nontrivially on the complement $\Omega_{k-1}$ of $\cup_{i=1}^{k-1} a_i$ as a segment of a cycle.

When choosing $a_k$ we require that again $\mathrm{supp}(a_k) \subset \Omega_{k-1}$ but this time we ask that $v_k:=[v_{k-1},a_k]$ moves exactly $2$ points of $\Omega_k:=\Omega_{k-1} \backslash \supp(a_k)$ and these are sent outside $\Omega_k$. To show that we can indeed choose such $a_k$  suppose that $v_{k-1}$ acts on the elements
$\{b_1, \ldots, b_t\} = \Omega_k$ as before: $b_j^{v_{k-1}}=b_{j+1}$ for $j=1, \ldots, t-1$. Note that $t>p_k+1$ since $t-p_k \in \{3,4,5\}$. Let $a_k$ be the $p_k$-cycle
$(b_1b_2 b_4 b_5b_6 \ldots b_{p_k +1})$, and put $\Omega_{k}= \Omega_{k-1} \backslash \mathrm{supp}(a_k)$. Then $[v_{k-1},a_k]=(b_1b_2b_{p_k+2})(b_3b_4b_5)$ fixes all elements of $\Omega_{k}$ except $b_{p_k+2}, b_3$ and sends $b_{p_{k}+2}$ to $b_1 \not \in \Omega_k$ and $b_3$ to $b_4 \not \in \Omega_k$.

Let us now denote by $\alpha, \beta$ the two elements from $\Omega_k$ moved by $v_k$ outside $\Omega_k$.
Suppose that $v_k (\alpha) = \gamma \not \in \Omega_k$ and $v_k(\beta)=\delta \not \in \Omega_k$. Recall that $|\Omega_k|$ is $3$,$4$ or $5$ by the choice of the sequence $p_i$. Choose an element  $\eta \in \Omega_k \backslash \{\alpha, \beta \}$ and put $a_0:=(\eta \alpha \beta)$.
Then $[a_0^{v_k},a_0]= [(\eta \gamma \delta),(\eta \alpha \beta)]= (\alpha \gamma \eta)$.
We have shown that there exists a 3-cycle $a_0$ with support in $\Omega_k$ such that $[v_k,a_0]$ is a $3$-cycle and if the starting full cycle $y'=v_0$ was even we are done. If $v_0$ was odd then replace $v_0$ by $v_0 \tau$ where $\tau$ is any transposition commuting with the $p_1$-cycle $a_1$. Such transposition exists because $p_1<n-2$ by assumption.
The lemma is proved.
\end{proof}

\medskip

We continue with the proof of Theorem~\ref{alt}.

Observe that if $w_1(x_1,x_2)$ is nontrivial then each $x_1^{m_i}$ must be nontrivial for $0 \leq i \leq k$ which means that $x_1$ has a cycle divisible by $p_i$ for each $i$. From $n- \sum_{i=0}^k p_i \in \{0,1,2\}$ and $p_i> (n- \sum_{j=1}^{i-1}p_j)/2$ it must be that $x_i$ has a single $p_i$-cycle for each $i$ and therefore each $x_1^{m_i}$ is a $p_i$-cycle. In particular $x_1^{m_0}$ is a $3$-cycle.

We conclude, that if $w_1(x_1,x_2) \not = 1$ then $x_1^{m_0}$ is a $3$-cycle. In this situation since $w_1(x_1,x_2)$ is a commutator of two $3$-cycles and since $w_1 \not =1$ we obtain that $w_1(x_1,x_2)$ is an even permutation with support of size at most $5$. We now take $w(x_1,x_2):=w_1(x_1,x_2)^{10}$, then $w$ is either trivial or a $3$-cycle. On the other hand Lemma~\ref{values} says that $w_1$ can take value some $3$-cycle and therefore, so can $w$. Theorem~\ref{alt} follows for $n \not =13$.

For $n=13$ we define $w_1=[(x_1^{M/5})^{[x_2,x_1^{M/7}]},x_1^{M/5}]$ and deduce as before that if $w_1 \not =1$ then $x_1^{M/7}$ is a $7$-cycle while $x_1^{M/5}$ is a $5$-cycle. Therefore $w_1$ is a nontrivial commutator of two 5-cycles and so has support of size at most $9$.
Direct inspection or an argument similar to Lemma~\ref{values} shows that $w_1$ takes value a product of a $5$-cycle and a $3$-cycle. Now consider $w_2=[(w_1^{8.5.7})^{[x_3,w_1^{8.9.7}]}, w_1^{8.5.7}]$.
If $w_2 \not = 1$ then $w_1^{8.9.7}$ is a $5$-cycle while $w_1^{8.5.7}$ can only be a $3$-cycle. Also by the same argument as Lemma~\ref{values} $w_2$ takes as values some (and hence all) $3$-cycles. Therefore $w_2$ has support of size at most $5$ and we take $w:=w_2^{10}$ as before.

When $n=7$ the argument is similar: Let
$w_1=[(x_1^{4.5.7})^{[x_2,x_1^{3.5.7}]},x_1^{4.5.7}]$. Again if $w_1 \not =1$ then
$x_1$ has cycles of size divisible by $3$ (since $x_1^{4.5.7} \not = 1$) and cycles of size divisible by 2 (since $x_1^{3.5.7} \not =1$). Now $x_1$ is even and so cannot be a 6-cycle. The only possibility is that $x_1^{4.5.7}$ is a double transposition and
$x_1^{4.5.7}$ is a $3$-cycle. Thus $w_1$ again has support of size at most $5$ and as before we check that $w_1$ takes some $3$-cycle as a value. Take $w=w_1^{10}$.

\begin{remark}
Notice that if $n \not=7$ then the word $w$ evaluated on $\Sym(n)$ takes value the identity and all $3$-cycles. For $n=7$ we take $w=w_1^{10}$ with $w_1=[(x_1^{4.5.7})^{[x_2,x_1^{2.3.5.7}]},x_1^{4.5.7}]$ and then a similar argument shows that $w$ takes only 3-cycles and the identity on $\Sym(7)$.
\end{remark}

\begin{remark}
In the symmetric group $\Sym(n)$ the conjugacy class of transpositions is smaller then the
conjugacy class of $3$-cycles. However, it is impossible to construct a word $w$ whose values are
either $1$ or  transpositions: If the image of $w$ is not inside $\Alt(n)$ then $w$ must have a free variable, say $x_1$ with odd exponent sum $m$. Setting $x_2=x_3= \cdots =1$ we see that the image of $w$ contains $x_1^m$ for any $x_1 \in \Sym(n)$ and in particular any involution is
a value of $w$.
\end{remark}

\begin{remark}
Using the word $w$ one can construct words $w_p$ for any prime $3 <p< n$ such that $\Alt(n)_{w_p}$
consist of the identity and all $p$-cycles. For any $k$ the support of $w'_k=(wz)^kz^{-k}$ is at most $3k$ since it is a product of $k$ $3$-cycles, moreover it can be a $2k+1$-cycle. Therefore $w_p = (w'_{(p-1)/2})^{N_p}$ is either $1$ or a $p$-cycle, where $pN_p$ is the exponent of $\Alt(3(p-1)/2)$.
\end{remark}

\section{Proof of Theorem~\ref{SL}}

Let $q,n>1$ be any integers. A prime $r$ is called a Zsigmondy prime for $q^n-1$ if $r$ divides $q^n-1$ but  does not divide $q^i-1$ for any $1 \leq i <n$. Zsigmondy's theorem states that Zsigmondy primes exist for all $q,n>1$ with the exceptions $n=6,q=2$ and $n=2, q+1=2^s$ some $s \in \N$.
A slight generalization~\cite{R} states that without any exceptions there is always a prime power $r^a$ such that
$r^\alpha| q^{n}-1$ but $r^\alpha$ does not divide $q^i-1$ for any $i < n$.
\begin{lemma}
\label{transvection}
Assume $n \not = 3,4$ or $n=3$ and $q \not = 2,4$. There exist integers $A,B$ such that for any $x\in \SL_n(q)$ if both $x^A$ and $x^B$ are not $1$ then $x^B$ is a transvection in $\SL_n(q)$ (and there exists $x$ such that $x^A$ and $x^B$ are not $1$).

When $n=4$ and $q \not =2,3$ there exist integers $A, \bar A,B$ such that
if $x^A,x^{\bar A},x^B \not =1$ then $x^B$ is a transvection.
\end{lemma}
\begin{proof}
Assume first that $n>4$.

Let $M$ be the exponent of $\SL_n(q)$ and let $r^\alpha$ be a Zsigmondy prime power for
$q^{n-2}-1$. Suppose $q=p^u$ some $u \in \N$ and let $M=A' r^a=B p^c$ with $A'$ coprime to $r$ and $B$ coprime to $p$ and set $A= A'r^{\alpha-1}$.
If both $x^A$ and $x^B$ are not trivial then both $r^\alpha$ and $p$ divide the order
of $x$. We can write $x=x'x''$ where $x'$ is a semisimple element while $x''$ is a unipotent element and both are in $\langle x \rangle$. Since $p$ divides the order of $x$, the element
$x''$ is non trivial and $x'$ has repeated eigenvalues.

We shall say that an eigenvalue $\lambda$ of a semisimple element $g$ generates $\F_{q^s}$ if $\F_q(\lambda)=\F_{q^s}$. The order of the semisimple element $g$ clearly divides the lowest common multiple of $(q^s-1)$ where $\F_{q^s}$ ranges over all fields generated by the eigenvalues of $g$.

Now, $r^\alpha$ divides the order of $x'$ and therefore
$x'$ has an eigenvalue generating $\F_{q^{c}}$ where $r^\alpha | q^c-1$.
The choice of $r$ implies that $c =n-2$.
Therefore, $x' \in T \times \GL_2(q)$ where $T$ is a torus of size $q^{n-2}-1$ in $\GL_{n-2}(q)$ and $x'$ can only have one repeated eigenvalue (inside $\F_q$) with multiplicity $2$. This gives that $x''$ acts as a Jordan block of size $2$ and $x^B = (x'')^B$ is a transvection.

The remaining cases $n=2,3,4$ are very similar.
The case $n=2$ is almost trivial -- any element in $\SL_2(q)$ of order $p$ is a transvection so we take $A=1$ and $B=M/p$.

When $n=3$ and $q \not =2,4$ take $r^\alpha$ to be any prime power which divides $q-1$ such that $r^\alpha\not=3$ and define $A$ and $B$ as in the case $n>4$. If $x^A\not=1$ and $x^B\not=1$ then $r^\alpha$ divides the order of $x'$ and $x'$ has
multiple eigenvalues -- this implies that all eigenvalues are in $\F_q$ and one has multiplicity $2$
and again $x''$ acts as a simple Jordan block of size $2$.

When $n=4$ and $q \not =2,3$ the same argument needs a small modification. Let $r^{\alpha}$ be a Zsigmondy prime power for $q^2-1$.
We define $A=M/r$ and $B=M/p^c$ as before but there is one additional case to be ruled out -- then $x'$
has two pairs of repeated eigenvalues which generate $\F_{q^2}$. Let $\overline {r}^{\bar \alpha} | q-1$ be a prime power different from $2$ and define $\bar{A}=M/\bar r$. Then
$x^{\bar A} \not=1$  implies  that the order of $x'$ is divisible by $\overline{r}^{\bar \alpha}$
which is not the case if $x'$ has two repeated eigenvalues.
\end{proof}

Using this lemma we can complete the proof of Theorem~\ref{SL}. Let $A, B$ be the integers specified by Lemma \ref{transvection}.
First consider the case $n > 4$.
Define the word  $w_1(x,y)=[x^B, (x^B)^{[y,x^{A}]}]$ -- if either $x^A$ or $x^B$ is
$1$ then so is $w_1$. Otherwise $w_1$ is a commutator of two conjugates of $x^B$.
In particular $w_1$ is a product of two transvection and so $w_1-\Id$ has rank at most 2. Hence at most two of the eigenvalues of $w_1$ can be different from 1 and so the semisimple part of $w_1$ must be contained in a conjugate of $\SL_2(q) \times \Id_{n-2}$. Further, if we suppose that
the unipotent part of $w_1$ is not 1 or a transvection, then (since $w_1 - \Id$ has rank at most 2) $w_1$ itself must be a unipotent Jordan block of size $3$ or a product of two unipotent Jordan blocks of size 2.
However a direct computation shows that such $w_1$ cannot be a commutator of two transvections.
This gives that in all cases $w=w_1^{(q-1)(q^2-1)}$ is a transvection or the identity.

It remains to show that $w$ is nontrivial, for which it is sufficient to show that $w_1$ can take a value some transvection. Let $e_1, \ldots, e_n$ be the standard basis of $\F_q^n$ and take $x=x'x''$ where $x''$ is a transvection in $\SL_2(q) \times \Id_{n-2}$ and $x'$ is a generator for the nonsplit torus in $\Id_2 \times \SL_{n-2}(q)$. Then $x^B$ is a transvection and $x^A$ has order $r$ and lies in
$\Id_2 \times \SL_{n-2}(q)$. To simplify notation put $x^A=a, x^B=b$. Without loss of generality we may assume that $b$ sends $e_1$ to $e_1+e_2$ and fixes all $e_j$ for $j >1$, i.e., $b$ acts as the elementary matrix $E_{1,2}(1)$.

We will show that there exists an element $y \in \SL_n(q)$ such that $e_1^{[y,a]}=e_2$ while $d:=e_2^{[y,a]} \in \langle e_3, \ldots ,e_n \rangle$. Now $b^{[y,a]}$ is a transvection such that $\mathrm{Im}(b^{[y,a]}- \Id)= \langle d \rangle$ and moreover $e_2^{b^{[y,a]}} =e_2+d$. Choose a scalar $\lambda \in F$ such that $e_1':=e_1 + \lambda e_2$ is fixed by $b^{[y,a]}$. Extend $d$ to a basis $\{d, d',d'', \ldots \}$ of $\ker (b - \Id) \cap \ker (b^{[y,a]} - \Id)$. With respect to the basis $\{e'_1,e_2,d,d',d'', \ldots \}$ of $\F_q^n$ the maps $b$ and
$b^{[y,a]}$ act as the matrices $E_{1,2}(1)$ and $E_{2,3}(1)$ respectively, hence $[b,b^{[y,a]}]=E_{1,3}(1)$ is a transvection.

To prove the existence of $y$ note that $a$ does not stabilize any proper subspace of $V:=\langle e_3, \ldots e_n\rangle$ because the order of $a$ is $r^\alpha$. Hence when $n\geq 5$ we may choose three linearly independent vectors vectors $t_1,t_2,t_3 \in V$ such that $t_1^{a^{-1}}=t_2, t_2^{a^{-1}}=t_3$. Define $y \in \SL_n(q)$ such that $e_i^ {y^{-1}}=t_i$, then we have that $e_1^{a^{-y}}=e_2$ and $e_2^{a^{-y}}=e_3$. Since $a$ fixes $e_2$ and sends $e_3$ to some vector $d \in V$ we have proved the existence of $y$ as claimed.
\medskip

The case $n=2$ is very easy because
the only element of order $p$ is a transvection, i.e., we can take $w = x^{q^2-1}$.

When $n=3$ and $q\not=2,4$ we can use the argument for $n >5$ but with the slightly modified choice of $A$ and $B$; when $q = 2$ or $q=4$ one can notice that the square of a unipotent element is a transvection,
therefore it suffices to take $w = x^{M/2}$ where $M$ is the exponent of the group.

When $n=4$ and $q\geq 4$ we consider $w_2=[x^B, (x^B)^{[[y,x^{A}],x^{\bar A}]}]$ -- the same argument shows that if $w_2\not =1$ then $x^B$ is a transvection. Similarly one can show the resulting word $w$ is not trivial.
For $q=3$, again, one uses that the cube of a unipotent element is a transvection, thus it suffices to
take $w=x^{M/3}$.

Finally when $n=4$, $q=2$ the square of a unipotent element is either trivial, a transvection or a product of two commuting transvections, so we can take $w = x^{2.3.5.7}$. In fact observing that $\SL_4(2) \simeq \Alt(8)$ and using Theorem~\ref{alt} we can find a word $w(x_1,x_2)$ whose image consists of the identity and double transvections only.


\begin{remark}
For $n\not=3,4$ the word $w$ constructed above also takes value the identity and all transvections on $\GL_n(q)$.
\end{remark}

\begin{remark}
Similar construction should be possible for other classical groups, but one needs to be more
careful with the possible structures of the unipotent elements.
\end{remark}
\medskip

\textbf{Acknowledgement} We thank the anonymous referee for pointing out an error in an earlier version of the paper.
The first autor was partiually supported by NSF grant DMS-0900932. The second author is supported by an EPSRC grant.

\noindent
M. Kassabov,\\
Cornell University, Ithaca, NY, USA\\
and\\
University of Southampton, Southampton, UK\\
\emph{email:} kassabov@math.cornell.edu

\bigskip

\noindent
N. Nikolov,\\
Imperial College, London, UK \\
\emph{email:} n.nikolov@imperial.ac.uk
\end{document}